\documentclass[11pt]{amsart}

\usepackage{amsthm}
\usepackage{amsmath}
\usepackage{amssymb}
\usepackage{color}

\newtheorem{thm}{Theorem}[section]
\newtheorem{theorem}[thm]{Theorem}

\newtheorem{problem}[thm]{Problem}

\newtheorem{definition}[thm]{Definition}
\theoremstyle{remark}
\newtheorem{remark}[thm]{Remark}

\begin{document}

\title{bases equivalent to the unit vector basis of $c_0$ or $\ell_p$}
\author[Casazza]{Peter G. Casazza}
\address{Department of Mathematics, University
of Missouri, Columbia, MO 65211-4100}

\thanks{The author was supported by
 NSF DMS 1906025}

\email{Casazzap@missouri.edu}

\begin{abstract}
We will show that an unconditional basis in a Banach space is equivalent to the unit vector basis of
$c_0$ or $\ell_p$ for $1\le p < \infty$ if and only if all finitely supported blocks  of the basis 
generated by a unit vector and its dual basis are uniformly
equivalent to the basis or all such blocks are uniformly complemented.
\end{abstract}

\maketitle

\section{Introduction}

There are several classical results concerning when a basis for a Banach space is equivalent to the unit vector basis
of $c_0$ or $\ell_p$ for $1\le p < \infty$.

\begin{definition}
Let $\{x_n\}_{n=1}^{\infty}$ be a basic sequence in a Banach space $X$. A sequence of non-zero vectors $\{u_j\}_{j=1}^{\infty}$ in X
of the form 
\[ u_j=\sum_{n=p_j+1}^{p_{j+1}}b_nx_n
\] with $\{b_n\}_{n=1}^{\infty}$ scalars and $p_1<p_2<\cdots$ an increasing sequence of
integers, is called a {\bf block basic sequence} or briefly a {\bf block basis} of $\{x_n\}_{n=1}^{\infty}$.
\end{definition}

\begin{definition}
If $\{x_n\}_{n=1}^{\infty}$ and $\{y_n\}_{n=1}^{\infty}$ are bases for Banach spaces $X,Y$ respectively, $\{x_n\}_{n=1}^{\infty}$ is 
{\bf K-equivalent} to $\{y_n\}_{n=1}^{\infty}$ if there is a constant $1\le K<\infty$ so that for all scalars $\{a_n\}_{n=1}^{\infty}$ we have
\[ \frac{1}{K}\|\sum_{n=1}^{\infty}a_ny_n\|\le |\sum_{n=1}^{\infty}a_nx_n\|\le K \|\sum_{n=1}^{\infty}a_ny_n\|.\]
Equivalently, if $T:X\rightarrow Y$ is given by $Tx_n=y_n$ then $\|T\|,\|T^{-1}\|\le K$. We say $\{x_n\}_{n=1}^{\infty}$ is {\bf equivalent}
to $\{y_n\}_{n=1}^{\infty}$ if they are K-equivalent for some K.
\end{definition}

In \cite{LT} we have:

\begin{theorem}
let $X$ be a Banach space with a normalized basis $\{x_n\}_{n=1}^{\infty}$. Then $\{x_n\}_{n=1}^{\infty}$ is equivalent
to the unit vector basis of $c_0$ or $\ell_p$, $1\le p < \infty$ if and only if $\{x_n\}_{n=1}^{\infty}$ is equivalent to
all of its normalized block bases.
\end{theorem}

\begin{definition}
A basis $\{x_n\}_{n=1}^{\infty}$ of a Banach space X is called {\bf subsymmetric} if it is unconditional and, for every
increasing sequence of integers $\{n_i\}_{i=1}^{\infty}$, $\{x_{n_i}\}_{i=1}^{\infty}$ is equivalent to $\{x_n\}_{n=1}^{\infty}$.
By switching to an equivalent norm we may assume the subsymmetric basis is 1-unconditional and 1-equivalent to all
of its subsequences.
\end{definition}

\begin{definition}
Let $\{x_n\}_{n=1}^{\infty}$ be a symmetric or subsymmetric basis for a Banach space X. Let $0\not= \alpha = \sum_{j=1}^{\infty}b_jx_j\in X$.
Given a subset of $\mathbb N$ $\{n_1<n_2< \ldots\}$ partition this set into infinitely many disjoint subsets $\{(i,j)\}_{i,j=1}^{\infty}$. Then the {\bf block
basis generated by $\alpha$} is $\{U_i^{\alpha}\}_{i=1}^{\infty}$ where
\[ U_i^{\alpha}=\sum_{j=1}^{\infty}b_jx_{ij}\mbox{ for all }i=1,2,\ldots.\]
\end{definition}

\begin{remark} \label{R1}
We have to be careful with this. In the symmetric case, for different $\{n_i\}_{i=1}^{\infty}$ and different partitions $\{(i,j)\}_{i,j=1}^{\infty}$
the $\{U_i^{\alpha}\}_{i=1}^{\infty}$ are symmetric and equivalent to each other.
But for subsymmetric bases, neither of these properties may hold.
\end{remark}

Also in \cite{LT} we have

\begin{theorem}\label{T2}
If $\{x_n\}_{n=1}^{\infty}$ is a symmetric basis for a Banach space X, then $\{x_n\}_{n=1}^{\infty}$ and its dual basis $\{x_n^*\}_{n=1}^{\infty}$ are equivalent
to all their block bases generated by a single vector if and only if $\{x_n\}_{n=1}^{\infty}$ is equivalent to the unit vector basis of $c_0$ or
$\ell_p$ for some $1\le p <\infty$.
\end{theorem}

Recently \cite{C} the subsymmetric version of this was proved.

\begin{theorem} If $\{x_n\}_{n=1}^{\infty}$ is a subsymmetric basis for a Banach space X then $\{x_n\}_{n=1}^{\infty}$ and $\{x_n^*\}_{n=1}^{\infty}$
are both equivalent to all block bases generated by a single vector if and only if $\{x_n\}_{n=1}^{\infty}$ is equivalent to the unit vector basis
of either $c_0$ or $\ell_p$ for some $1\le p < \infty$.
\end{theorem}

\begin{remark}
In the theorems above it is important that we work with possibly infinitely supported vectors since the theorems
fail if we can only work with finitely supported vectors as we will see in the next section. 
\end{remark}

In this paper we look at finitely supported block bases generated by a single vector and prove the corresponding
results for this family but need the additional requirement that the equivalences are uniform.

\section{Bases Equivalent to the unit vector basis of $c_0$ or $\ell_p$}

\begin{definition}
Let $\{x_n\}_{n=1}^{\infty}$ and $0\not= \alpha =\sum_{n=1}^mb_nx_n$. The (finitely supported)
{\bf block basis generated by $\alpha$} is the block basis
$\{U_i^{\alpha}\}_{i=1}^{\infty}$ where
\[ U_i^{\alpha}=\sum_{n=(i-1)m+1}^{im}b_{n-km}x_n,\mbox{ for all }i=1,2,\ldots.\]
\end{definition}

For subsymmetric bases, blocks generated by a single vector are always equivalent to the basis, but as we will see, they
may not be uniformly equivalent.

\begin{theorem}
Let $\{x_n\}_{n=1}^{\infty}$ be a subsymmetric basis for a Banach space X. Assume $0\not= \alpha = \sum_{n=1}^mb_nx_n\in X$ and
$\{U_i^{\alpha}\}_{i=1}^{\infty}$ is the block basis generated by $\alpha$. Then $\{U_i^{\alpha}\}_{i=1}^{\infty}$ is equivalent to 
$\{x_n\}_{n=1}^{\infty}$.
\end{theorem}

\begin{proof}
We may assume that the basis is 1-unconditional and 1-equivalent to all of its subsequences. Assume $b_j\not= 0$ for some fixed
$1\le j \le m$.
We compute for any $\sum_{n=1}^{\infty}a_nx_n$:
\begin{align*}
\|\sum_{n=1}^{\infty}a_nx_n\|&=\|\sum_{i=1}^{\infty}a_ix_{(i-1)m+j}\|\\
&=\frac{1}{b_j}\|\sum_{i=1}^{\infty}a_ib_jx_{(i-1)m+j}\|\\
&\le \frac{1}{b_j}\|\sum_{i=1}^{\infty}a_iU_i^{\alpha}\|\\
&\le \frac{1}{b_j}\sum_{k=1}^mb_k\|\sum_{i=1}^{\infty}a_ix_{(i-1)m+k}\|\\
&=\frac{1}{b_j}\left ( \sum_{k=1}^mb_k\right ) \|\sum_{i=1}^{\infty}a_ix_i\|.
\end{align*}
This completes the stated equivalence.
\end{proof}

Now we can prove the main result of the paper.

\begin{theorem}
Let $\{x_n\}_{n=1}^{\infty}$ be a unconditional basis for a Banach space X. The following are equivalent:
\begin{enumerate}
\item $\{x_n\}_{n=1}^{\infty}$ is equivalent to the unit vector basis for $c_0$ or $\ell_p$ for $1\le p < \infty$.
\item There is a constant $1\le K$ so that whenever $\alpha = \sum_{n=1}^mb_nx_n$ is a unit vector, the block basis generated
by $\alpha$ is K-equivalent to $\{x_n\}_{n=1}^{\infty}$, the the same result holds for $\{x_n^*\}_{n=1}^{\infty}$.
\item There is a constant $1\le K$ so that whenever $\alpha = \sum_{n=1}^mb_nx_n$ is a unit vector then the block basis
$\{U_i^{\alpha}\}_{i=1}^{\infty}$ generated by $\alpha$ is complemented in the space by a projection $P$ with $\|P\|\le K$,
and the same is true for $\{x_n^*\}_{n=1}^{\infty}$.
\end{enumerate}
\end{theorem}

\begin{proof}
$(1)\Rightarrow (2), \ (3)$: This is obvious.
\vskip12pt

$(3)\Rightarrow (2)$: We may assume the basis is 1-unconditional.
Choose norm one vectors $v_1=\sum_{i=1}^mb_ix_i$ and $w_1=\sum_{i=m+1}^nc_ix_i$ and let $u_1=v_1+w_1$. Define
$u_i=v_i+w_i$ for $i=1,2,\ldots$ by
\[v_i=\sum_{j=(i-1)n+1}^{(i-1)n+m}b_{j-(i-1)n}x_i\mbox{ and }w_i=\sum_{j=(i-1)n+m+1}^{in}c_{j-(i-1)n}x_i.\]
By assumption, there is a projection $Q$ onto the block basis $\{u_i\}_{i=1}^{\infty}$ generated by the vector $u_1$ with $\|Q\|\le K$.
For each $i=1,2,\ldots$ assume
\[ Qv_i=\sum_{j=1}^{\infty}b_{ij}u_j\mbox{ and }Qw_i = \sum_{j=1}^{\infty}c_{ij}u_j.\]
Then $b_{ii}+c_{ii}=1$ for all $i=1,2,\ldots$. Define an operator $T:[Qv_i]\rightarrow [w_i]$ by $T(Qv_i)=w_i$ for all $i=1,2,\ldots$. So $\|TQ\|\le K$ and the matrix of $T$ is
$[b_{ij}+c_{ij}]$. And since the basis is 1-unconditional, the norm of the diagonal operator $D$ of $TQ$ is $\|D\|\le \|TQ\|$. But the 
diangonal operator maps $v_i$ to $w_i$. So for all scalars $\{a_i\}_{i=1}^{\infty}$ we have
\[ \|\sum_{i=1}^{\infty}a_iw_i\|=\|\sum_{i=1}^{\infty}a_iTv_i\|\le K \|\sum_{i=1}^{\infty}a_iv_i\|.\]
Switching the roles of $v_i$ and $w_i$ we have for all scalars $\{a_i\}_{i=1}^{\infty}$
\[ \|\sum_{i=1}^{\infty}a_iv_i \|\le K\|\sum_{i=1}^{\infty}a_iw_i\|\]
So $\{v_i\}_{i=1}^{\infty}$ is equivalent to $\{w_i\}_{i=1}^{\infty}$.
Letting $v_1=x_1$, we discover that $\{w_i\}_{i=1}^{\infty}$ is equivalent to $\{x_i\}_{i=1}^{\infty}$.

The same proof works for $\{x_n^*\}_{n=1}^{\infty}$.

\vskip12pt
$(2)\Rightarrow (1)$: We will do this in steps. For any $n\in \mathbb{N}$ define
\[ \lambda(n)=\|\sum_{i=1}^nx_i\|\mbox{ and }\mu(n)=\|\sum_{i=1}^nx_i^*\|.\]
\vskip12pt
\noindent {\bf Step 1}: We show that either $X$ is isomorphic to $c_0$ or $\ell_1$ or there is a $1<p$ so that for all $n\in \mathbb{N}$
\[\frac{1}{K}\le \frac{\lambda(n)}{n^{1/p}}\le K.\]
\vskip12pt
For this, fix $n\in \mathbb{N}$ and let $\alpha_n=\frac{1}{\lambda(n)}\sum_{i=1}^nx_i$.
Consider the block basis generated by $\alpha_n$ as
\[ U_i^{\alpha_n}=\frac{1}{\lambda(n)}\sum_{j=(i-1)n+1}^{in}x_i,\mbox{ for all }i=1,2,\ldots.\]
By (2), for every $m=1,2,\ldots$ we have
\[ \frac{1}{K}\lambda(m)\le \| \sum_{i=1}^mU_i^{\alpha_n}\| = \frac{1}{\lambda_n}\lambda(mn)\le K\lambda(m).
\]
Hence,
\[\frac{1}{K}\le \frac{\lambda(mn)}{\lambda(n)\lambda(m)}\le K.\]
By the proof of \cite{LT} Theorem (2,a,9) either $sup_n\lambda(n)<\infty$, or we have for some $1\le p < \infty$,
\begin{equation}\label{E5} \frac{1}{K}\le \frac{\lambda(n)}{n^{1/p}} \le K,\mbox{ for all }n=1,2,\ldots.
\end{equation}
If $sup_n\lambda_n<\infty$ then X is isomorphic to $c_0$ and if $p=1$ then X is isomorphic to $\ell_1$ by \cite{LT} the remark 
following (3,a,7). So we assume $1<p<\infty$ and choose $1<q$ so that $\frac{1}{p}+\frac{1}{q}=1$.
\vskip12pt
\noindent {\bf Step 2}: We have for all $n\in \mathbb{N}$
\[ \frac{1}{2K}\le \frac{\mu(n)}{n^{1/q}}\le 2K.\]
\vskip12pt
For this, we apply \cite{LT} Proposition (3,a,6), page 116:
\[ n\le \lambda(n)\mu(n)\le 2n.\]
Hence
\[1 \le \frac{\lambda(n)\mu(n)}{n}=\frac{\lambda(n)}{n^{1/p}}\frac{\mu(n)}{n^{1/q}}\le 2.
\]
The result is now immediate from Equation \ref{E5}.
\vskip12pt
\noindent {\bf Step 3}: For all scalars $\{a_i\}_{i=1}^{\infty}$ we have
\[ \left \|\sum_{i=1}^{\infty}a_ix_i\right \|\le 2K\left ( \sum_{i=1}^{\infty}|a_i|^{p}\right )^{1/p}.\]
The same is true for all $\{a_i\}_{i=1}^{\infty}$
\[ \|\sum_{i=1}^{\infty}a_ix_i^*\|\le 2K\left ( \sum_{i=1}^{\infty}|a_i|^q\right )^{1/q}.\]
\vskip12pt
For this, we can copy line by line \cite{LT} Theorem (3,a,10), page121.
\vskip12pt
\noindent {\bf Step 5}:We conclude the proof.
\vskip12pt
By steps (2), (3), and (4) we have for all scalars $\{a_i\}_{i=1}^{\infty}$
\[ \left \|\sum_{i=1}^{\infty}a_ix_i^*\right \|\le 2K\left ( \sum_{i=1}^{\infty}|a_i|^{q}\right )^{1/q}.\]
By duality, for all $\{a_i\}_{i=1}^{\infty}$ we have
\[ \left \|\sum_{i=1}^{\infty}a_ix_i\right \|\ge \frac{1}{2K}\left ( \sum_{i=1}^{\infty}|a_i|^{p}\right )^{1/p}.\]

This concludes the proof of the theorem.
\end{proof}

\begin{remark}
We would hope to not have to assume the basis is unconditional since (3) implies that every subsequence of the basis
is uniformly complemented. Unfortunately, this property does not imply unconditionality of the basis as we see below.
\end{remark}
\noindent {\bf Example}:
Lwt $\{x_i\}_{i=1}^{\infty}$ be the row vectors of the matrix:
\[ \begin{bmatrix}
1&0&0&0&0&\cdots\\
1&1&0&0&0&\cdots\\
1&1&1&0&0&\cdots\\
\vdots&\vdots&\vdots&\vdots&\vdots&
\end{bmatrix} \]
The set $\{x_i\}_{i=1}^{\infty}$ is called the summing basis of $c_0$. We will now show that every subsequence of
this basis is 1-complemented. Fix natural numbers $n_0=0<n_1<n_2< \cdots$. For $x=\sum_{i=1}^{\infty}a_ix_i$ 
in terms of the unit vector basis this vector is:
\[ x=\left (\sum_{i=1}^{\infty}a_i,\sum_{i=2}^{\infty}a_i,\sum_{i=3}^{\infty}a_i,\ldots\right ),\]
and so the norm of this vector is:
\[ \|x\|=max_{i=1}^{\infty}|\sum_{j=i}^{\infty}a_j|.
\]
Define a projection $P$ by:
\[ P(x)=\sum_{i=0}^{\infty} \left ( \sum_{j=n_i+1}^{n_{i+1}}a_j\right )x_{n_i}.\]
It is clear that this is a projection and in terms of the unit vector basis this is
\[Px=\left (\sum_{i=1}^{\infty}a_i,\ldots,\sum_{i=1}^{\infty}a_i(n_1),\sum_{i=2}^{\infty}a_i(n_1+1),
\ldots,\sum_{i=2}^{\infty}a_i(n_2),
\sum_{i=3}^{\infty}a_i(n_2+1),\ldots,\sum_{i=3}^{\infty}a_i(n_3),\ldots 
\right )
\]
Hence,
\[ \|Px\|=max_{i=1}^{\infty}|\sum_{j=i}^{\infty}a_j|=\|x\|.\]

An obvious question is:

\begin{problem}
Do we need the conditions in Theorem to include both $\{x_n\}_{n=1}^{\infty}$ and $\{x_n^*\}_{n=1}^{\infty}$ or is just
one of these enough?
\end{problem}

\end{document}